\documentclass[12pt]{amsart}
\usepackage{amssymb,amsmath,amsfonts,latexsym}
\usepackage{bm}
\newcommand{\newsection}[1]{\setcounter{equation}{0} \section{#1}}
\newcommand{\bea}{\begin{eqnarray}}
\newcommand{\eea}{\end{eqnarray}}

\newcommand{\clh}{\mathcal{H}}

\newcommand{\raro}{\rightarrow}

\def \qed {\hfill \vrule height6pt width 6pt depth 0pt}
\def\textmatrix#1&#2\\#3&#4\\{\bigl({#1 \atop #3}\ {#2 \atop #4}\bigr)}
\def\dispmatrix#1&#2\\#3&#4\\{\left({#1 \atop #3}\ {#2 \atop #4}\right)}
\newcommand{\be}{\begin{equation}}
\newcommand{\ee}{\end{equation}}
\newcommand{\ben}{\begin{eqnarray*}}
\newcommand{\een}{\end{eqnarray*}}

\newcommand{\NI}{\noindent}

\newcommand{\bi}{\begin{itemize}}
\newcommand{\ei}{\end{itemize}}

\newcommand{\dsc}{\mathbb{D}}
\newcommand{\cplx}{\mathbb{C}}

\newcommand\ba[1]{\overline{#1}}

\newtheorem{Theorem}{\sc Theorem}[section]
\newtheorem{Lemma}[Theorem]{\sc Lemma}
\newtheorem{Proposition}[Theorem]{\sc Proposition}
\newtheorem{Corollary}[Theorem]{\sc Corollary}
\newtheorem{Definition}[Theorem]{\sc Definition}
\newtheorem{Example}[Theorem]{\sc Example}
\newtheorem{Remark}[Theorem]{\sc Remark}
\newtheorem{Note}[Theorem]{\sc Note}
\newtheorem{Question}{\sc Question}
\newtheorem{ass}[Theorem]{\sc Assumption}
\newcommand{\bt}{\begin{Theorem}}
\def\beginlem{\begin{Lemma}}
\def\beginprop{\begin{Proposition}}
\def\begincor{\begin{Corollary}}
\def\begindef{\begin{Definition}}
\def\beginexamp{\begin{Example}}
\def\beginrem{\begin{Remark}}
\def\beginq{\begin{Question}}
\def\beginass{\begin{ass}}
\def\beginnote{\begin{Note}}
\newcommand{\et}{\end{Theorem}}
\def\endlem{\end{Lemma}}
\def\endprop{\end{Proposition}}
\def\endcor{\end{Corollary}}
\def\enddef{\end{Definition}}
\def\endexamp{\end{Example}}
\def\endrem{\end{Remark}}
\def\endq{\end{Question}}
\def\endass{\end{ass}}
\def\endnote{\end{Note}}

\begin{document}

\title[Characterizations of Symmetrized Polydisc]{Characterizations of Symmetrized Polydisc}

\author[Gorai]{Sushil Gorai}
\address{Indian Statistical Institute, Statistics and Mathematics Unit, 8th Mile, Mysore Road, Bangalore, 560059, India}
\email{sushil.gorai@gmail.com}

\author[Sarkar]{Jaydeb Sarkar}
\address{Indian Statistical Institute, Statistics and Mathematics Unit, 8th Mile, Mysore Road, Bangalore, 560059, India}
\email{jay@isibang.ac.in, jaydeb@gmail.com}



\begin{abstract}
Let $\Gamma_n$, $n \geq 2$, denote the symmetrized polydisc in
$\mathbb{C}^n$, and $\Gamma_1$ be the closed unit disc in
$\mathbb{C}$. We provide some characterizations of elements in
$\Gamma_n$. In particular, an element $(s_1, \ldots, s_{n-1}, p) \in
\mathbb{C}^n$ is in $\Gamma_n$ if and only if $s_j = \beta_j +
\overline{\beta_{n-j}} p$, $j = 1, \ldots, n-1$, for some $(\beta_1,
\ldots, \beta_{n-1}) \in \Gamma_{n-1}$, and $|p| \leq 1$.
\end{abstract}

\maketitle

\newsection{Introduction}

In this paper we study and characterize the symmetrized polydisc in
the $n$-complex plane $\mathbb{C}^n$, $n \geq 2$. We denote by
$\Gamma_n := \overline{\mathbb{G}_n}$ the symmetrized polydisc in
$\mathbb{C}^n$, where \[\mathbb{G}_n = \{\pi_n (\bm{z}) : \bm{z} \in
{\mathbb{D}}^n\},\]$\mathbb{D} = \{z \in \mathbb{C} : |z| < 1\}$,
and  $\pi_n : \mathbb{C}^n \raro \mathbb{C}^n$ is the symmetrization
map
\[\pi_n(\bm{z}) = (\sum_{1 \leq i \leq n} z_i, \sum_{1 \leq i_1 <
i_2 \leq n} z_{i_1} z_{i_2}, \ldots, \prod_{i=1}^n z_i),\]for all
$\bm{z} = (z_1, \ldots, z_n) \in \mathbb{C}^n$. It follows easily
from the fundamental theorem of algebra that $\pi_n$ is an onto map,
that is, $\pi_n(\mathbb{C}^n) = \mathbb{C}^n$.

We first turn to the case of symmetrized bidisc $\Gamma_2$. The
early study of symmetrized bidisc was motivated to a large extent by
the theory spectral Nevanlinna-Pick problem, $\mu$-synthesis problem
in control engineering and complex geometry of domains in
$\mathbb{C}^n$ (cf. \cite{AY4}, \cite{AY6}, \cite{AY7}). One of the
most important tools in the investigation of function theory and
operator theory on $\Gamma_2$ is the following classification result
due to Agler and Young: Let $(s, p) \in \mathbb{C}^2$ and $|p| \leq
1$. Then
\[(s, p) \in \Gamma_2 \Leftrightarrow  \exists \beta \in
\overline{\mathbb{D}} \mbox{~such that~} s = \beta + p
\overline{\beta}.\]Moreover, for a pair of commuting operators $(S,
P)$ acting on some separable Hilbert space $\clh$, $\Gamma_2$ is a
spectral set of $(S, P)$ (see \cite{AY5}) if and only if $S = X + P
X^*$ for some operator $X$ with some natural contractivity property
in an appropriate sense (see \cite{S}).

In this paper we establish some characterization results concerning
the elements of $\Gamma_n$, $n \geq 2$. One of our main results is the
following theorem.

\bt\label{T:Gamman} Let $(s_1,\dots,s_{n-1},p)\in\cplx^n$, $n \geq
2$, and $\Gamma_1 = \overline{\mathbb{D}}$. Then the following are
equivalent: \bi
\item[(i)] $(s_1,\dots,s_{n-1},p)\in \Gamma_n$.
\smallskip

\item[(ii)]$|p| \leq 1$ and $s_j=\beta_j+\ba{\beta_{n-j}}p$, $j=1,\dots, n-1$, for some $(\beta_1,\dots,\beta_{n-1})\in\Gamma_{n-1}$.
\ei \et

\NI In particular, for the case of $\Gamma_2$ we recover the
characterization result of Agler and Young.

For the rest of the paper, we let $n \geq 2$ and denote the unit
circle $\{z \in \mathbb{C}: |z| = 1\}$ by $\mathbb{T}$. Given a
holomorphic function $f$ on a domain $\Omega \subseteq
\mathbb{C}^m$, we denote by $Z_f$ the set of all zeroes of $f$ in
$\Omega$.

The reminder of this paper is organized as follows. Section 2 below
collects some general results and facts concerning zero sets of
polynomials. The proofs of the main theorems are given in Section 3.

\vspace{0.2in}

\noindent \textsf{Added in proof:} After completing our work, we became aware of the work of Constantin Costara (see \cite{CC} and \cite{CC-th}) which overlaps considerably with ours. In particular, one of our main results, Theorem \ref{T:Gamman}, have already appeared in Costara's work (see \cite{CC-th} and Theorem 3.7  in \cite{CC}). However, our proof of Theorem \ref{T:Gamman} is simple and different from that of Constantin Costara. Moreover, Theorem \ref{T:chargnkernel} and the equivalence of (ii) and (iii) of Theorem \ref{T:chargn} are new. We are indebted to Nicholas Young and Constantin Costara for pointing out these references to us.

\newsection{Preparatory Results}\label{sec-technical}

In this section we will prove a couple of auxiliary results that
will be used in the proofs of the theorems. We will also list some
results from the literature about the location of zeros of
polynomials.

\begin{Lemma}\label{L:gn}
Let $(s_1,\dots, s_{n-1},p)\in\cplx^n$ with $|p|<1$. Then there is
an element $(\beta_1,\dots, \beta_{n-1})\in \cplx^{n-1}$ such that $
s_j=\beta_j+\ba{\beta_{n-j}}p$, $j=1,\dots, n-1$. Moreover,
$(\beta_1,\dots, \beta_{n-1})$ is given by the following identity\[
 \beta_j=\frac{s_{n-j}-\ba{s_j}p}{1-|p|^2},\qquad j=1,\dots, {n-1}.
\]
\end{Lemma}

\begin{proof} The conclusion amounts to solve the following set of
equations:
 \[
 s_j=\beta_j+\ba{\beta_{n-j}}p,\qquad j=1, \ldots, {n-1},
 \]
for $(\beta_1, \dots,
 \beta_{n-1})\in \mathbb{C}^{n-1}$.
We now consider the above $n-1$ equations in pairs, that is:
\begin{equation}\label{eq-sk}
 s_k = \beta_k + \ba{\beta_{n-k}}p \quad \quad \mbox{and} \quad
 \quad s_{n-k} = \beta_{n-k}+\ba{\beta_k} p,
\end{equation}
where $k=1, \dots,[n/2]$, if $n$ is odd, and $k=1, \dots, n/2$, if
$n$ is even. Since $|p|<1$, by solving each pair of equations in
\eqref{eq-sk}, we get that
\[
 \beta_j=\frac{s_{n-j}-\ba{s_j}p}{1-|p|^2},\qquad j=1,\dots, {n-1}.
\]
This completes the proof.
\end{proof}

The following lemma is an application of the Rouche's theorem. This
will play an important role in our approach to the symmetrized
polydisc.

\begin{Lemma}\label{L:polyzero}
Let $n \geq 2$, $(a_1, \ldots, a_n) \in \mathbb{C}^n$, $|a_n| <1$,
and $(b_1, \ldots, b_{n-1}) \in \mathbb{C}^{n-1}$. Furthermore, let
$a_j=b_j+\overline{b_{n-j}}a_n$, $j=1,\dots, n-1$, and
$f(z)=z^n-a_1z^{n-1}+\dots + (-1)^{n-1}a_{n-1}z+(-1)^na_n$ and
 $g(z)=z^{n-1}-b_1z^{n-2}+\dots +
 (-1)^{n-2}b_{n-2}z+(-1)^{n-1}b_{n-1}$. Then

\NI (i) The number of zeros of $f$ in $\dsc$ is the same as the
number of zeros of $zg$ in $\dsc$.

\NI(ii) $Z_f \cap \mathbb{T} = Z_g \cap \mathbb{T}$.
\end{Lemma}

\begin{proof} For $a_j= b_j+\ba{b_{n-j}}a_n$, $j=1,\dots, {n-1}$, as above we have
\[\begin{split}
 f(z)& = z^n-a_1z^{n-1}+\dots + (-1)^{n-1}a_{n-1}z+(-1)^na_n \\
 & = z^n- (b_1+\ba{b_{n-1}}a_n)z^{n-1}+\dots +(-1)^{n-1}(b_{n-1}+\ba{b_1}a_n)z+ (-1)^na_n.
\end{split}\]
It follows that
\begin{equation}\label{E:zerosfg}
 |f(z)-zg(z)|=|a_n||b_{n-1}\ba{z}^{n-1}-\dots +(-1)^{n-2}b_1\ba{z}+(-1)^{n-1}|.
\end{equation}
If we restrict the above equation on $\mathbb{T}$, we get
\[
|f(z)-zg(z)|=|a_n||z^{n-1}-b_1z^{n-2}+\dots+(-1)^{n-2}b_{n-2}z+(-1)^{n-1}b_{n-1}|.
\]
By virtue of $|a_n| <1$ this yields
\[|(f-zg)(w)| =|a_n||g(w)| = |a_n| |(zg)(w)| <|(zg)(w)|,
\]for all $|w| = 1$.  Then Rouche's theorem shows that $f$ and $zg$ have the same number of zeroes inside $\dsc$.
This completes the proof of $(i)$

\NI We now turn to (ii). Let $f(\lambda)=0$ for some $\lambda \in
\mathbb{T}$. Hence, by \eqref{E:zerosfg}
\[
|g(\lambda)|=|a_n||g(\lambda)|.
\]
If $g(\lambda)\neq 0$ then we have $|g(\lambda)|<|g(\lambda)|$,
which is a contradiction. This implies that $g(\lambda)=0$.

\NI Conversely, suppose $g(\lambda)=0$ for some $\lambda \in
\mathbb{T}$. Using \eqref{E:zerosfg} again it follows that
\[
|f(\lambda)|=0.
\]
This completes the proof of (ii).
\end{proof}

Now, as an easy consequence we obtain:

\begin{Proposition}\label{Prop:zeros}
Let $n \geq 2$, $(a_1, \ldots, a_n) \in \mathbb{C}^n$, $|a_n| <1$,
and $(b_1, \ldots, b_{n-1}) \in \mathbb{C}^{n-1}$. Furthermore, let
$a_j=b_j+\overline{b_{n-j}}a_n$, $j=1,\dots, n-1$, and
$f(z)=z^n-a_1z^{n-1}+\dots + (-1)^{n-1}a_{n-1}z+(-1)^na_n$ and
 $g(z)=z^{n-1}-b_1z^{n-2}+\dots +
 (-1)^{n-2}b_{n-2}z+(-1)^{n-1}b_{n-1}$.  Then all the zeros of $f$ lies
 in $\ba{\dsc}$ if and only if all the zeros of $g$ lies in $\ba{\dsc}$.
 \end{Proposition}

\begin{proof}
The proof follows directly from Lemma~\ref{L:polyzero}.

\end{proof}

Among the applications to domains in $\mathbb{C}^n$ in which
positive definite forms played an important role, the following one,
which will be useful in the sequel, stand out as particularly
impressive \cite{Schur} (see also \cite{PY}):

\bt\label{T:schur}\textsf{(Schur)} Given a polynomial
$p(z)=a_0z^n+a_1z^{n-1}+\dots +a_n$, $a_0\neq 0$, the zero set $Z_p$
is a subset of $\dsc$ if and only if the Hermitian form
\[ H(x)= \sum_{j=1}^n |\overline{a_0}x_j+\overline{a_1}x_{j+1}+\dots +\overline{a_{n-j}}x_n|^2
-\sum_{j=1}^n|a_nx_j+a_{n-1}x_{j+1}+\dots +a_{n-j}x_n|^2
\]
is positive definite. \et

\qed

\newsection{Main Results}\label{sec-proofs}

With this background in place, we now state and prove the main
results of this paper. Some of our results are new even in the case
of $\Gamma_2$.

We first note that a necessary condition for $\bm{w} \in
\mathbb{C}^n$ to be in $\mathbb{G}_n$ is that $|w_j|<
{{n}\choose{j}}$, $j = 1, \ldots, n$. Let us assume 
$\bm{w}=\pi_n(\bm{z})$ for some $\bm{z}\in\mathbb{C}^n$. It follows from the definition of $\mathbb{G}_n$
that: {\em if $\pi(\bm{z})\in \mathbb{G}_n$, for $\bm{z} =
(z_1,\dots, z_n) \in \mathbb{C}^n$, then $\pi(z_1,\dots, z_j) \in \mathbb{G}_j$, for any $1 < j<n$.}
Hence, it is also necessary that, for each $1< k<n$ and $1\leq m_1<\dots< m_k\leq n$
\begin{equation}\label{E:wmk}
(w^{(k)}_{m_1},\dots,w^{(k)}_{m_k})=\pi_k(z_{m_1},\dots,z_{m_k}),
\end{equation}
 we have
$|w^{(k)}_{m_l}|<{{k}\choose{l}}$ for all $1\leq l\leq k$.
We are now ready for the first
characterization result.

\begin{Theorem}\label{T:chargnkernel}
Let $\bm{z}, \bm{w} \in \mathbb{C}^n$, $\pi_n(\bm{z}) = \bm{w}$ and
$|w_j|< {{n}\choose{j}}$, $j = 1, \ldots, n$.  
Also assume that $|w^{(2)}_{m_1}|<2$ and $|w^{(2)}_{m_2}|<1$ for all $1\leq m_1< m_2\leq n$, 
where $w^{(2)}_{m_1}, w^{(2)}_{m_2}$ are as in the notation  of \eqref{E:wmk}.
Then $\bm{w} \in
\mathbb{G}_n$ if and only if
\[ \mathop{\Pi}_{j=1}^n(1-|z_j|^2)>0.\]
\end{Theorem}

\begin{proof}
Let $\bm{z}, \bm{w} \in \mathbb{C}^n$, $\pi_n(\bm{z}) = \bm{w}$ and
$|w_j|< {{n}\choose{j}}$, $j = 1, \ldots, n$.

\NI If $\mathop{\Pi}_{j=1}^n(1-|z_j|^2)>0$, then $|z_j|<1$ for some
$1\leq j\leq n$. If not, then $|\Pi_{j=1}^nz_j|>1$, which
contradicts the fact that $|w_n| < {{n}\choose{n}} = 1$. We claim
that $z_ i \in \mathbb{D}$ for all $i = 1, \ldots, n$. If not, then
$|z_l|>1$ for some $l$, $1\leq l\leq n$. From this and the fact that
$\mathop{\Pi}_{j=1}^n(1-|z_j|^2)>0$, it readily follows that $|z_m|
>1$ for some $m = 1, \ldots, n$ and $l \neq m$. We may assume without loss of generality that $l=1$ and $m=2$. That is,
\begin{equation}\label{E:kerdim2}
|z_1|>1 \quad \text{and}\quad |z_2|>1.
\end{equation}
 We now consider $(w^{(2)}_1,w^{(2)}_2)=\pi(z_1,z_2)$. By equation \eqref{E:kerdim2} 
we get that $|w^{(2)}_2|>1$, which contradicts the assumption that $|w^{(2)}_2|<1$.

\NI The necessary part follows immediately from the definition of
$\mathbb{G}_n$.
\end{proof}

The following result (see \cite{AY8} and \cite{AY5}) provides a
useful way to characterize the elements of $\mathbb{G}_2$: Let $(s,
p) \in \mathbb{C}^2$. Then $(s, p) \in \mathbb{G}_2$ if and only if
\[\begin{bmatrix} 1- |p|^2 & - \bar{s} + s \bar{p}\\- s + \bar{s} p
& 1 - |p|^2\end{bmatrix} > 0,\]if and only if \[s = \beta + p
\bar{\beta},\]for some $\beta \in \mathbb{G}_1 := \mathbb{D}$.

We generalize the above fact in the following sense:

\bt \label{T:chargn} Let $(s_1, s_2, \dots, s_{n-1}, p)\in \cplx^n$
with $|p|<1$. Then the following are equivalent.

\NI (i) $(s_1, s_2, \dots, s_{n-1}, p)\in \mathbb{G}_n$.

\NI (ii) There exists $(\beta_1, \dots, \beta_{n-1})\in
\mathbb{G}_{n-1}$ such that $s_j=\beta_j+\ba{\beta_{n-j}}p$, $j=1,
\dots, n-1$.
\smallskip

\NI(iii) The following matrix is positive definite:
\[
\begin{bmatrix}
1-|p|^2 & -s_1+\ba{s}_{n-1}p&  \dots & (-1)^{n-1}s_{n-1}+ (-1)^n \ba{s}_1 p \\
-\ba{s}_1 + s_{n-1}\ba{p} & 1+|s_1|^2-|s_{n-1}|^2-|p|^2 & \dots& (-1)^{n-2}s_{n-2}+(-1)^{n-1}\ba{s}_2p\\
\vdots & \vdots &\dots & \vdots \\
(-1)^{n-1}\ba{s}_{n-1}+(-1)^ns_1\ba{p} &
(-1)^{n-2}\ba{s}_{n-2}+(-1)^{n-1}s_2\ba{p} & \dots & 1-|p|^2
\end{bmatrix}.
\]
\et

\begin{proof}
We first prove $(i)\Longleftrightarrow (ii)$. Since
$(s_1,\dots,s_{n-1},p)\in \cplx^n$ with $|p|<1$, by Lemma~\ref{L:gn}
there exists $(\beta_1,\dots,\beta_{n-1})\in \cplx^{n-1}$ such that
\[
s_j=\beta_j+\overline{\beta_{n-j}}p,\qquad j=1,\dots, {n-1}.
\]
By Proposition~\ref{Prop:zeros}, $(s_1,\dots,s_{n-1}, p)\in
\mathbb{G}_n$ if and only if $(\beta_1,\dots,\beta_{n-1})\in
\mathbb{G}_{n-1}$.

$(i)\Longleftrightarrow (iii)$ follows from Theorem~\ref{T:schur},
by unfolding the positive definiteness of the Hermitian form in term
of the Hermitian matrix that it corresponds to.
\end{proof}

To proceed further, we recall the following results on the
distinguished boundary of $\Gamma_n$ (see Theorem 2.4, \cite{BsR}).
Recall also that the distinguished boundary of $\Gamma_n$ is given
by $\{\pi_n(\bm{z}) : \bm{z} \in \mathbb{T}^n\}$.

\bt \label{T:distinguishedbdy} Let $(s_1,\dots,s_{n-1},p)\in
\cplx^n$. Then the following statements are equivalent:

\NI (i) $(s_1,\dots,s_{n-1},p)$ lies in the distinguished boundary
of $\Gamma_n$.

\NI (ii) $|p|=1$ and $s_j=\beta_j+\ba{\beta_{n-j}}p$, $j=1,\dots,
{n-1}$ for some $(\beta_1, \dots, \beta_{n-1})$ in the distinguished
boundary of $\Gamma_{n-1}$. \et

We now proceed to prove the main theorem.

\begin{proof}[Proof of Theorem~\ref{T:Gamman}]

Note that by virtue of Theorem \ref{T:distinguishedbdy} stated
above, we only have to consider the case $|p| < 1$.

\NI Let $(s_1,\dots,s_{n-1},p)\in \Gamma_n$. First assume that
$|p|<1$. Invoke Lemma~\ref{L:gn} to conclude that
\begin{equation}\label{eq-sj}
s_j=\beta_j+\ba{\beta_{n-j}}p,\quad j=1,\dots,n-1,
\end{equation}
for some $(\beta_1,\dots,\beta_{n-1})\in \cplx^{n-1}$. Consider $f,
g \in \mathbb{C}[z]$ as $\beta_1,\dots,\beta_{n-1}$:
\begin{align*}
f(z)&:=z^n-s_1z^{n-1}+\dots+(-1)^{n-1}s_{n-1}z+(-1)^np\\
g(z)&:=z^{n-1}-\beta_1z^{n-2}+\dots+(-1)^{n-1}\beta_{n-1}.
\end{align*}
Since $(s_1,\dots,s_{n-1},p)\in \Gamma_n$, $Z_f\subset \ba{\dsc}$.
Then by Proposition~\ref{Prop:zeros} it follows that all the zeros
of $g$ lies in $\ba{\dsc}$. Hence, $(\beta_1,\dots, \beta_{n-1})\in
\Gamma_{n-1}$.

\NI Conversely, suppose $(ii)$ holds.  Let $(s_1,\dots,
s_{n-1},p)\in \cplx^n$ and there exists
$(\beta_1,\dots,\beta_{n-1})\in\Gamma_{n-1}$ such that the equations
in \eqref{eq-sj} hold. As before, we only need to treat the case
$|p|<1$. We will again apply Proposition~\ref{Prop:zeros} to  on the
polynomials:
\begin{align*}
f(z)&:=z^n-s_1z^{n-1}+\dots+(-1)^{n-1}s_{n-1}z+(-1)^np\\
g(z)&:=z^{n-1}-\beta_1z^{n-2}+\dots+(-1)^{n-1}\beta_{n-1}.
\end{align*}
We conclude from here that $(s_1,\dots,s_{n-1},p)\in \Gamma_n$.
\end{proof}

\noindent\textbf{Acknowledgment:} The first named author is grateful
to Indian Statistical Institute, Bangalore Centre for warm
hospitality. The work of the first named author is supported by an
INSPIRE faculty fellowship (IFA-MA-02) funded by DST.

\end{document}